\documentclass[12pt]{amsart}
\usepackage{amssymb}
\usepackage{enumerate,color}
\usepackage{graphicx}
\usepackage[text={6in,9in},centering]{geometry}
\usepackage{subfigure}

\usepackage{tikz}
\usetikzlibrary{positioning, decorations.pathreplacing, calc}
\usepackage{amsmath, amssymb}

\theoremstyle{plain}
\newtheorem{thm}{Theorem}[section]
\newtheorem{cor}[thm]{Corollary}
\newtheorem{lem}[thm]{Lemma}
\newtheorem{prop}[thm]{Proposition}
\theoremstyle{definition}

\theoremstyle{remark}
\newtheorem{rem}{Remark}

\def\R{\mathbb{R}}

\newcommand{\bu}{\mathbf{u}}
\newcommand{\bv}{\mathbf{v}}

\newcommand{\w}{\omega} 
\newcommand{\be}{\begin{equation}} 
	\newcommand{\ee}{\end{equation}} 

\newcommand\pkd{\mbox{\rm dim}_{\rm P}\,} 
\newcommand\hdd{\mbox{\rm dim}_{\rm H}\,} 
\newcommand\ubd{\overline{\mbox{\rm dim}}_{\rm B}\,} 
\newcommand\lbd{\underline{\mbox{\rm dim}}_{\rm B}\,} 

\newcommand{\red}[1]{{\color{red}#1}}

\title[]{Generalized $q$-dimensions of measures on  Non-autonomous conformal sets}

\date{}

\author[J. J. Miao]{Jun Jie Miao}
\address[J. J. Miao]{School of Mathematical Sciences,  Key Laboratory of MEA(Ministry of Education) \& Shanghai Key Laboratory of PMMP,  East China Normal University, Shanghai 200241, P.~R. China}
\email{jjmiao@math.ecnu.edu.cn}

\author[Tianrui Wang]{Tianrui Wang}
\address[Tianrui Wang]{School of Mathematical Sciences,  Key Laboratory of MEA(Ministry of Education) \& Shanghai Key Laboratory of PMMP,  East China Normal University, Shanghai 200241, P.~R. China}
\email{51265500036@stu.ecnu.edu.cn}

\begin{document}
	\begin{abstract}
We study the generalized $q$-dimensions of measures supported on  non-autonomous conformal attractors,  which are the generalizations of  Moran sets and the attractors of iterated function systems. We first prove that the critical values of generalized upper and lower pressure functions are always the upper bounds for the upper and lower generalized $q$-dimensions of  measures supported on non-autonomous conformal sets. Then we obtain dimension formulas for generalized $q$-dimensions if non-autonomous conformal attractors satisfy  certain separation conditions, and moreover, the generalized $q$-dimension formulae may be simplified for the Bernoulli measures. Finally, we provide the generalized $q$-dimension formulae for measures supported on autonomous conformal sets.
	\end{abstract}
	
	\maketitle

\section{Introduction}
\subsection{Generalized $q$-dimensions.}
The generalized $q$-dimensions of compactly supported Borel probability measures are important concepts in  fractal geometry and dynamical system which were introduced by R\'{e}nyi in \cite{A.R, A.R.} in the 1960s. They  quantify the global fluctuations of a given measure $\nu$ and provide valuable information about the multifractal properties of $\nu$ and also about the dimensions of its support. 

Let $\nu$ be a positive finite Borel measure on $\mathbb{R}^d$. For $q\not=1$, the {\it lower} and {\it upper generalized $q$-dimensions} of $\nu$ are given by
\begin{equation}\label{Dq}
\underline{D}_q(\nu)=\liminf_{\delta\to 0}\frac{\log \sum_{Q\in{\mathcal{M}_\delta}}\nu(Q)^q}{(q-1)\log \delta},  \qquad   \overline{D}_q(\nu)=\limsup_{\delta\to 0}\frac{\log \sum_{Q\in{\mathcal{M}_\delta}}\nu(Q)^q}{(q-1)\log \delta},
\end{equation}
where $\mathcal{M}_\delta$ is the family of $\delta$-mesh cubes in $\mathbb{R}^d$. For $q=1$, $\underline{D}_1(\nu)$ and $\overline{D}_1(\nu)$ are defined by
\begin{equation}\label{D1}
\underline{D}_1(\nu)=\liminf_{\delta\to 0}\frac{ \sum_{{Q\in\mathcal{M}_\delta}}\nu(Q)\log\nu(Q)}{\log \delta},  \quad   \overline{D}_1(\nu)=\limsup_{\delta\to 0}\frac{ \sum_{{Q\in\mathcal{M}_\delta}}\nu(Q)\log\nu(Q)}{\log \delta},
\end{equation}
If $\underline{D}_q(\nu)=\overline{D}_q(\nu)$, we write ${D}_q(\nu)$ for the common value which we refer to as the generalized $q$-dimension. Note that the generalized $q$-dimension of a measure contains information about the measure and its support. It directly follows from the definition that
$$
\ubd \mbox{spt}(\nu)=\overline{D}_0(\nu),  \qquad  \lbd \mbox{spt}(\nu)=\underline{D}_0(\nu),
$$
where {\it spt} denotes the support of $\nu$. We refer readers to \cite{Fal, YB} for  background reading.

The generalized $q$-dimension is very important in dimension theory, and it has a wide range of applications. Shmerkin \cite{S} has computed the generalized $q$-dimensions of dynamically driven self-similar measures, which are a class of non-autonomous similar measures supported on non-autonomous similar sets, and has used them to prove Furstenberg's long-standing conjecture on the dimension of the intersections of $\times p$- and $\times q$-invariant sets, stating that the mappings $T_p$ and $T_q$ are strongly transverse.

A concept intimately related to the generalized $q$-dimension is the $L^q$-spectrum, which is closely related to various key concepts in fractal geometry. For $q \neq 1$, it is defined by $\tau_q(\nu) = (1-q)D_q(\nu)$. There is a rich literature concerning measures supported on fractal sets. It was shown by Peres and Solomyak \cite{BS} that the $L^q$ spectrum of any self-conformal measure exists for $q > 0$.   Fraser \cite{JMF}   extended it to graph-directed self-similar measures and self-affine measures.  In \cite{PN}, for a conformal iterated function system satisfying the strong open set condition, Patzschke used pressure function to  determine a precise formula for  $D_q(\mu)$ of self-conformal measure $\mu$. Miao and Wu \cite{MWu}  studied the generalized $q$-dimension of measures on the Heisenberg group. We refer the reader to \cite{Falco05,Falco10,FFL,FX,F,JMF,K,SMN,SX,YYL} for further related works.

It is then natural to explore the properties of generalized $q$-dimensions for measures supported on non-autonomous attractors. However, unlike classical attractors generated by iterated function systems, such sets generally lack dynamical invariance, rendering the powerful tools of ergodic theory inapplicable. Consequently, the existence of generalized $q$-dimensions is less common in this setting, and determining their properties becomes significantly more challenging. Some progress has been made by Gu and Miao \cite{GM2}, who provided formulas for the generalized $q$-dimensions of non-autonomous similar measures and  non-autonomous affine measures. 

	\subsection{Non-autonomous conformal iterated function systems.}

Non-autonomous iterated function systems may be regarded as a generalization of iterated function systems. First, we recall the definitions of  non-autonomous iterated function systems.

Let $\{I_{k}\}_{k\geq 1}$ be a sequence of finite index sets with $\# I_k\geq 2$. Given integers $k\ge l \ge 1$,  we write
\begin{equation}\label{finite I}
\Sigma_l^{k}=\{u_{l}u_{l+1}\ldots u_{k}: u_{j}\in I_j, j=l,l+1,\ldots, k \},
\end{equation}
and for simplicity, we set $ \Sigma^k=\Sigma^{k}_1 $ if $l=1$. We write $\Sigma^{*}=\bigcup _{k=0}^{\infty }\Sigma^{k}$ for the set of all finite words with $
\Sigma^{0}=\{\emptyset \}$ containing only the empty word $\emptyset$. We write
\begin{equation}\label{def_infSeqS}
\Sigma^{\infty}=\{\mathbf{u}=u_{1}u_{2}\ldots u_{k}\ldots : u_{k}\in I_k ,\  k=1,2,\ldots \}
\end{equation}
for the set of words with infinite length.

Let $J\subset\mathbb{R}^d$  be a compact set with non-empty interior, and $J$ satisfies $\overline{\mbox{int}(J)}=J$. For each integer $k>0$, let $\Phi_k=\{\varphi_{k, i} \}_{i\in I_k}$ be a family of mappings $\varphi_{k, i}:J\to J$. We say the collection $\mathcal{J}=\{J_{\mathbf{u}}:\mathbf{u}\in \Sigma^*\}$ of closed subsets of $J$ fulfils the \textit{non-autonomous structure with respect to $\{\Phi_k\}_{k=1}^\infty$} if it satisfies the following conditions:
\begin{itemize}
\item[(i).] There exists $0<c<1$ such that for all integer $k>0$ and all $i\in I_k$,
\begin{equation} \label{def_uccdn}
|\varphi_{k, i}(x)-\varphi_{k, i}(y)|\le  c|x-y| \qquad \textit{ for all }  x, y\in J.
\end{equation}
\item[(ii).] For all integers $k>0$ and all $\mathbf{u}\in \Sigma^{k-1}$, the elements $J_{\mathbf{u}i}, i\in I_k$ of $\mathcal{J}$ are the subsets of $J_{\mathbf{u}}$. We write $J_{\emptyset }=J$ for the empty word $\emptyset $.

\item[(iii).] For each $\mathbf{u}=u_1\ldots u_k \in \Sigma^*$,  there exists $\w_{\bu} \in \R^d$ and $\Psi_\mathbf{u}: \mathbb{R}^{d}\rightarrow \mathbb{R}^{d}$  such that
\begin{equation}\label{basic set}
    J_{\mathbf{u}}=\Psi_{\mathbf{u}}(J)=\varphi_\bu(J)+\omega_\bu,
   \end{equation}
    where $\varphi_\bu=\varphi_{u_1}\circ \cdots \circ \varphi_{u_j} \cdots \circ \varphi_{u_k}  $ and $\varphi_{u_j}\in \Phi_j$.
\end{itemize}
We call  $\boldsymbol{\Phi}=\{\Phi_k\}_{k=1}^\infty$   a \emph{non-autonomous iterated function system} and
\begin{equation}\label{att}
E=E(\boldsymbol{\Phi})=\bigcap_{k=1}^\infty \bigcup_{\mathbf{u}\in\Sigma^k}\Psi_{\mathbf{u}}(J)
\end{equation}
the \emph{non-autonomous attractor} of $\boldsymbol{\Phi}$. If for all integers $k\ge 0$ and all $\mathbf{u}\in\Sigma^{k},$  
\begin{equation}\label{def_OSC}
\mbox{int}(J_{\mathbf{u}i})\cap \mbox{int}(J_{\mathbf{u}j})=\emptyset \qquad \textit{for all $i\not= j\in I_{k+1}$,}
\end{equation}
 we say that $E$ satisfies the \emph{open set condition (OSC)}.
If \eqref{def_OSC} is replaced by 
$$
J_{\mathbf{u}i}\cap J_{\mathbf{u}j}=\emptyset, \qquad \textit{for all $i\not= j\in I_{k+1}$,}
$$
we say that $E$ satisfies the \emph{strong separation condition (SSC)}.

Let $\boldsymbol{\Phi}$ be a non-autonomous iterated function system satisfying that
\begin{itemize}
\item[(iv).] There exists an open connected set $V$ independent of $k$ with $J\subset V$ such that each $\varphi_{k, j}$ extends to a $C^1$ conformal diffeomorphism of $V$ into $V$.
\item[(v).] There exists a constant $C\ge 1$ such that for all $\bu=u_lu_{l+1}\ldots u_k\in\Sigma_l^k$ and   all $x, y\in V$.
$$
\|D\varphi_{\bu}(x)\|\le C\|D\varphi_{\bu}(y)\|,
$$ 
where $D\varphi_\bu(x)$ is the derivative of $\varphi_\bu$  at $x$.
\end{itemize}
We say $\boldsymbol{\Phi}=\{\Phi_k\}_{k=1}^\infty$ is a \emph{non-autonomous conformal iterated function system (NCIFS)}, and we call its attractor $E$ the \emph{non-autonomous conformal set} of $\boldsymbol{\Phi}$.  Let $\|D\varphi_\bu\| =\sup\{\|D\varphi_\bu(x)\| : x \in J\}$  and
\begin{equation}\label{Mk}
 M_k=\max_{\mathbf{u} \in \Sigma^k} \{\|D\varphi_\mathbf{u}\|\}, \quad  \quad
\underline c_k=\min_{1\le j \le \# I_k} \{\|D\varphi_{k, j}\|\}.
\end{equation}
If $\Phi_k$ only consists of similarities for all $k\geq 1$, and the corresponding attractor $E$ satisfies the open set condition, then $E$ is called a Moran set. 

There has been a large amount of literature on the dimension theory of non-autonomous fractals \cite{GM, MW,RM}. In particular, Moran sets are a typical case of non-autonomous conformal sets, and under the assumption 
\begin{equation}\label{condition}
\lim_{k\to +\infty} \frac{\log\underline c_k}{\log M_k}=0,
\end{equation}
where $\underline{c}_k=\min_{1\le j \le n_k} \{c_{k,j}\},$ and $M_k=\max_{\mathbf{u} \in \Sigma^k} \lvert J_\mathbf{u} \rvert$, Hua, Rao, Wen and Wu in \cite{HRWW} proved that the dimension formulas of Moran set are given by
	\begin{equation*}
		\hdd E=s_{\ast }=\liminf_{m\rightarrow \infty }s_{m}, \quad \pkd E=\ubd E=s^{\ast }=\limsup_{m\rightarrow \infty }s_{m}.
	\end{equation*}
where $s_{k}$ is the unique real solution
	of the equation $\prod\nolimits_{i=1}^{k} \sum\nolimits_{j=1}^{n_{i}}(c_{i,j})^{s}=1$.  We refer readers to  \cite{GM,Hua,RM, Wen00} for details and related works. 

In this paper, we study the  generalized $q$-dimension of measures supported on non-autonomous conformal iterated function systems $\boldsymbol{\Phi}$ satisfying open set condition and \eqref{condition} which plays a fundamental role in the study  of Moran fractals.

\subsection{Symbolic space and Pressure functions}
Let $\Sigma^{k}=\Sigma_1^{k}$ and $\Sigma^\infty$ be given by \eqref{finite I} and \eqref{def_infSeqS}, respectively.
We  topologize $\Sigma^\infty$ using the metric
$d(\mathbf{u},\mathbf{v})=2^{-|\mathbf{u}\wedge  \mathbf{v}|}$ for distinct $\mathbf{u},\mathbf{v} \in \Sigma^\infty$ to make $\Sigma^\infty$ into a compact metric space. For each $\bu \in\Sigma^\infty$, we write $\bu|_n=u_1\ldots u_n$.  For each $\mathbf{u}=u_1 \dots u_k \in \Sigma^k$, we write $\mathbf{u}^* =u_1 \dots u_{k-1}$.  Given $\mathbf{u}\in \Sigma^l$, for $\mathbf{v}\in \Sigma^k$ where $k\geq l$ or $\mathbf{v}\in\Sigma^\infty $,  we write $\mathbf{u}\prec \mathbf{v}$ if $u_i =v_i$ for all $i=1,2,\ldots, l$.  We define the \textit{cylinders} $[\bu]=\{\mathbf{v}\in \Sigma^\infty : \mathbf{u}\prec \mathbf{v}\}$ for $\mathbf{u}\in \Sigma^*$; the set of cylinders $\{[\bu] : \mathbf{u} \in \Sigma^* \}$ forms a base of open and closed neighbourhoods for $\Sigma^\infty$. We term a subset $\mathcal{C}$ of $\Sigma^*$ a
\textit{cut set} if $\Sigma^\infty\subset\bigcup_{\mathbf{u}\in \mathcal{C}}[\bu]$, where $[\bu]\cap[\bv]=\emptyset$ for  all $\mathbf{u}\neq \mathbf{v}\in \mathcal{C}$. It is equivalent to that, for every $\mathbf{w}\in \Sigma^\infty$, there is a unique word $\mathbf{u}\in \mathcal{C}$ with $|\mathbf{u}|<\infty$ such that $\mathbf{u}\prec \mathbf{w}$. Given a cut set $\mathcal{C}$, we write 
$$
k_{\mathcal{C} }=\min\{|\mathbf{u}|:\mathbf{u}\in\mathcal{C}\}.
$$

We define the projection mapping $\pi_{\boldsymbol{\Phi}}:\Sigma^\infty\to J$ by $ \pi_{\boldsymbol{\Phi}}(\mathbf{u})=\bigcap_{n=1}^\infty J_{\mathbf{u}|_n}.$ 
Alternatively we may write the non-autonomous set of $\boldsymbol{\Phi}$ as
\begin{equation}\label{att_1}
E=E(\boldsymbol{\Phi})=\pi_{\boldsymbol{\Phi}}(\Sigma^\infty).
\end{equation}
Given a positive finite Borel measure $\mu$ on $\Sigma^\infty$, the image measure  $\mu^\omega$  of $\mu$ given by
\begin{equation}\label{mu}
\mu^\omega(A)=\mu\{\mathbf{u}:\pi_{\boldsymbol{\Phi}}(\mathbf{u})\in A\}\qquad  \textit{for $A\subset \mathbb{R}^d$}
\end{equation}
is a Borel measure supported on the non-autonomous conformal set $E$. We say $\mu^\omega$ satisfies the \emph{bounded overlap condition (BOC)} if there exists a constant $C\ge 1$ such that for each $\bu\in\Sigma^*$
\begin{equation}\label{condition ssc}
C^{-1}\mu^\omega(J_\bu)\le\mu([\bu])\le C\mu^\omega(J_\bu),
\end{equation}

Given a sequence of probability vectors $\{\mathbf{p}_k=(p_{k, 1}, \ldots, p_{k, \#I_k})\}_{k=1}^\infty$, that is,  for each  $k>0$,  $\sum_{i=1}^{\#I_k}p_{k, i}=1$, for each cylinder $[\bu]$, we   define $\mu$ on $\Sigma^\infty$ by setting
\begin{equation}\label{b m}
\mu([\bu])=p_\mathbf{u}=p_{1, u_1}p_{2, u_2}\cdots p_{k, u_k},
\end{equation}
 and extend it  to  a measure on $\Sigma^\infty$ in the usual way. We call the corresponding  projection measure $\mu^\omega$ on $E$ given by \eqref{mu}  a \emph{ non-autonomous conformal measure}.

Given a NCIFS $\boldsymbol{\Phi}$ and a positive finite Borel measure $\mu$ on $\Sigma^\infty$. For  $\delta>0$, let
$$
\mathcal{C}(\delta)=\{\mathbf{u}\in\Sigma^*: \|D\varphi_\mathbf{u}\|\le \delta< \|D\varphi_{\mathbf{u}^*}\| \},
$$
and for simplicity, we write $k_{\delta}= k_{\mathcal{C}(\delta)}=\min\{|\mathbf{u}|:\mathbf{u}\in\mathcal{C}(\delta)\}.$
  For $t\in \mathbb{R}$, we define generalized upper and lower pressure functions of $\mu$ respectively by    
\begin{equation}\label{def_GULPF}
\begin{split}
&\overline P_\mu(t, q)=\left\{ \begin{aligned} &\limsup_{\delta\to 0} \frac{\mathrm{sgn}(1-q)}{k_{\delta}}\log \sum_{\mathbf{u}\in\mathcal{C}(\delta)}\|D\varphi_\mathbf{u}\|^{t(1-q)}\mu([\bu])^q ,  \textit{  $q>0$, $q\not=1$}   \\
&\limsup_{\delta\to 0} -\frac{1}{k_{\delta}} \sum_{\mathbf{u}\in\mathcal{C}(\delta)}\mu([\bu])\log(\|D\varphi_\mathbf{u}\|^{-t}\mu([\bu])), \hspace{1cm}q=1,
\end{aligned} \right.   \\
&\underline P_\mu(t, q)=\left\{ \begin{aligned} & \liminf_{\delta\to 0} \frac{\mathrm{sgn}(1-q)}{k_{\delta}}\log \sum_{\mathbf{u}\in\mathcal{C}(\delta)}\|D\varphi_\mathbf{u}\|^{t(1-q)}\mu([\bu])^q ,  \textit{  $q>0$, $q\not=1$}   \\
&\liminf_{\delta\to 0} -\frac{1}{k_{\delta}} \sum_{\mathbf{u}\in\mathcal{C}(\delta)}\mu([\bu])\log(\|D\varphi_\mathbf{u}\|^{-t}\mu([\bu])), \hspace{1cm} q=1.
\end{aligned} \right.
\end{split}
\end{equation}    
If $\overline P_\mu(t, q)=\underline P_\mu(t, q)$, we call the common value, denoted by $P_\mu(t, q)$, the generalized pressure function.
We write their jump points respectively as
\begin{equation}\label{d*}
\begin{split}
\overline{d}_q^*=\inf\{t:\overline P_\mu(t, q)<0\}=\sup\{t:\overline P_\mu(t, q)>0\},  \\
\underline{d}_q^*=\inf\{t:\underline P_\mu(t, q)<0\}=\sup\{t:\underline P_\mu(t, q)>0\}.
\end{split}
\end{equation}             
Note that $\overline{d}_q^*$ and $\underline{d}_q^*$ may be infinite, and their existence  is established by Lemma \ref{m d}.

\subsection{Main conclusions}

Generally, it is difficult to find generalized $q$-dimensions of measures supported on non-autonomous conformal sets, but we are still able to provide some rough estimates under certain conditions. First, we show that $\underline{d}_q^*$ and $\overline{d}_q^*$ are  natural upper bound for lower and upper generalized $q$-dimensions.

\begin{thm}\label{upper bound}

Let $\Phi$ be a NCIFS satisfying (1.10) and $\mu$ a positive finite Borel measure on $\Sigma^{\infty}$. Let $\mu^{\omega}$ be the image measure of $\mu$. Then for all $q>0$,
$$
\underline{D}_q(\mu^\omega)\le\min\{\underline{d}_q^*, d\},  \qquad  \overline{D}_q(\mu^\omega)\le\min\{\overline{d}_q^*, d\},
$$
where $\underline{d}_q^*$ and $\overline{d}_q^*$ are given by \eqref{d*}.
\end{thm}

Due to the geometric properties of non-autonomous conformal sets, we are able to find the generalized $q$-dimensions formulas under OSC.
\begin{thm}\label{Dq E}
Let $\Phi$ be a NCIFS satisfying the open set condition and (1.10) and $\mu$ a positive finite Borel measure on $\Sigma^{\infty}$.  Let $\mu^\omega$ be the image measure of $\mu$ satisfying the bounded overlap condition. Then for all $q>0$,
$$
\underline{D}_q(\mu^\omega)=\underline{d}_q^*,  \qquad  \overline{D}_q(\mu^\omega)=\overline{d}_q^*,
$$
where $\underline{d}_q^*$ and $\overline{d}_q^*$ are given by \eqref{d*}.
\end{thm}

\begin{rem}
If NCIFS $\boldsymbol{\Phi}$ satisfies the SSC, then the condition \eqref{condition ssc} holds.
\end{rem}

Alternatively, one can often work with a simpler form of the pressure function defined directly on the sequence of levels $k$. We define upper and lower pressure functions of $\mu$ respectively by  
\begin{equation}\label{def_ULPF}
\begin{split}
\overline P^\mu(t, q)=\limsup_{k\to \infty} \mbox{sgn}(1-q)\frac{1}{k}\log \sum_{\mathbf{u}\in\Sigma^k}\|D\varphi_\mathbf{u}\|^{t(1-q)}\mu([\bu])^q ,  \textit{  $q>0$, $q\not=1$}   
\\
\underline P^\mu(t, q)= \liminf_{k\to \infty} \mbox{sgn}(1-q)\frac{1}{k}\log \sum_{\mathbf{u}\in\Sigma^k}\|D\varphi_\mathbf{u}\|^{t(1-q)}\mu([\bu])^q ,  \textit{  $q>0$, $q\not=1$}   
\end{split}
\end{equation}    
If $\overline P^\mu(t, q)=\underline P^\mu(t, q)$, we call the common value $P^\mu(t, q)$  the pressure function.
We write their jump points respectively as
\begin{equation}\label{dl}
\begin{split}
\overline{d}_q=\inf\{t:\overline P^\mu(t, q)<0\}=\sup\{t:\overline P^\mu(t, q)>0\},  \\
\underline{d}_q=\inf\{t:\underline P^\mu(t, q)<0\}=\sup\{t:\underline P^\mu(t, q)>0\}.
\end{split}
\end{equation}  

For $q>0$ and $q\neq 1$, we are able to show that the generalized $q$-dimensions are given by $\underline{d}_q$ and $\overline{d}_q$. 
\begin{thm}\label{simple}
Given NCIFS $\boldsymbol{\Phi}$ satisfying OSC and \eqref{condition}.  Let $\mu$ be a  positive  finite Borel measure on $\Sigma^\infty$. If the image measure $\mu^\omega$  of $\mu$ satisfies the bounded overlap condition, then for $q>1$,
$$
\underline{D}_q(\mu^\omega)=\underline{d}_q=\underline{d}_q^*,
$$
and for $0<q<1$,
$$
\overline{D}_q(\mu^\omega)=\overline{d}_q =\overline{d}_q^*.
$$
\end{thm}

Given NCIFS $\boldsymbol{\Phi}=\{\Phi_k\}_{k=1}^\infty$, if $\Phi_k=\Phi_1$ for all $k\geq 1$, we call  $\boldsymbol{\Phi}$ an \emph{autonomous conformal iterated function system}, and $E$ an \emph{autonomous conformal set}. Note that non-autonomous iterated function systems may also be regarded as a generalization of the iterated  function system. However, for each $\bu\in\Sigma^\infty$, the choice of translations $\Psi_\bu$ in a non-autonomous iterated function system is very flexible. Given an autonomous conformal iterated function system $\{\Phi\}$, the corresponding autonomous conformal set may be not a self-conformal set. See Figure \ref{fig:nonautonomous_set}. For an autonomous conformal set, if all translations  $\omega_\bu$ are $0$, then the system reduces to a standard (autonomous) conformal IFS, and its attractor is a self-conformal set.

\begin{figure}[htbp]\label{fig:nonautonomous_set}
\centering
\begin{tikzpicture}[x=13cm, y=-0.9cm, scale=0.9]

\draw[very thick] (0,0) -- (1,0);
\node[above] at (0.5,0) {$J = [0,1]$};

\draw[very thick] (0,1) -- (0.5,1);   
\draw[very thick] (0.6,1) -- (0.933,1); 
\node[above] at (0.25,1) {$\Psi_1(J)=J_1$};
\node[above] at (0.7665,1) {$\Psi_2(J)=J_2$};

\draw[very thick] (0.05,2) -- (0.216,2);   
\draw[very thick] (0.226,2) -- (0.476,2); 
\draw[very thick] (0.61,2) -- (0.776,2); 
\draw[very thick] (0.79,2) -- (0.901,2); 
\node[above] at (0.133,2) {$J_{12}$};
\node[above] at (0.351,2) {$J_{11}$};
\node[above] at (0.693,2) {$J_{21}$};
\node[above] at (0.8455,2) {$J_{22}$};

\draw[very thick] (0.05,3) -- (0.105,3);     
\draw[very thick] (0.11,3) -- (0.193,3); 
\draw[very thick] (0.226,3) -- (0.309,3);  
\draw[very thick] (0.331,3) -- (0.456,3); 
\draw[very thick] (0.61,3) -- (0.693,3);  
\draw[very thick] (0.738,3) -- (0.776,3); 
\draw[very thick] (0.8,3) -- (0.837,3); 
\draw[very thick] (0.84,3) -- (0.895,3); 
\node[above] at (0.0775,3) {$J_{122}$};
\node[below] at (0.1515,3) {$J_{121}$};
\node[above] at (0.2675,3) {$J_{112}$};
\node[below] at (0.3935,3) {$J_{111}$};
\node[above] at (0.6515,3) {$J_{211}$};
\node[below] at (0.757,3) {$J_{212}$};
\node[above] at (0.8185,3) {$J_{222}$};
\node[below] at (0.8675,3) {$J_{221}$};

\end{tikzpicture}

\caption{Autonomous conformal set  at level 1, 2, 3, with   $J=[0,1]$ and $\Phi_k=\{\varphi_{1}=\frac{1}{2}x, \varphi_{2}=\frac{1}{3}x\}$.}
\label{fig:nonautonomous_set} 
\end{figure}
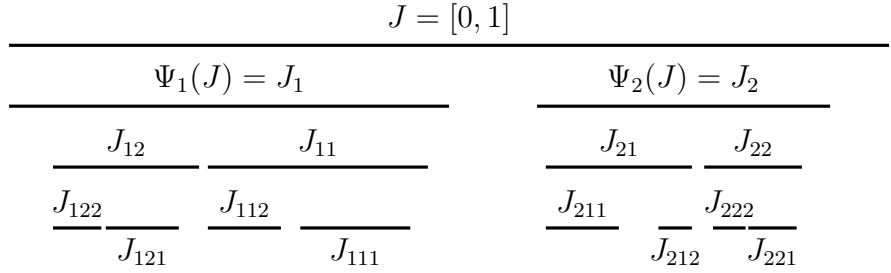

For  $\mathbf{u}=u_1u_2 \ldots\in\Sigma^\infty $, define the shift map $\sigma$ by $\sigma(u_1 u_2 \ldots)=u_2 u_3 \ldots$.  For $f:\Sigma^\infty\to\mathbb{R}$,  we write 
$$
S_kf(\mathbf{u})=\sum_{j=0}^{k-1}f(\sigma^j(\mathbf{u})).  
$$
A Borel probability measure $\mu$ on $\Sigma^\infty$ is a Gibbs measure if there exists a continuous $f:\Sigma^\infty\to\mathbb{R}$,  a number $P(f)$ termed the pressure of $f$ and $a>0$  such that
\begin{equation}\label{Gibbs}
a\le\frac{\mu([\mathbf{u}|_k])}{exp(-kP(f)+S_kf(\mathbf{u}))}\le\frac{1}{a},
\end{equation}
 for all $\mathbf{u} \in\Sigma^\infty$ and all $k$. Thus the pressure is given by
$$
P(f)=\lim_{k\to\infty}\frac{1}{k}\log \sum_{\mathbf{u}\in\Sigma^k}e^{S_kf(\mathbf{v})},
$$
where $\mathbf{v}$ is any element of $\Sigma^\infty$ such that $\mathbf{v}|_k=\mathbf{u}$. By the variational principle, if $f$ satisfies $|f(\mathbf{u})-f(\mathbf{v})|\le cd(\mathbf{i},\mathbf{j})^\epsilon$ for some $\epsilon>0$,  then there exists an invariant Gibbs measure $\mu$ satisfying \eqref{Gibbs} for some $P(f)$.  See Bowen \cite{Bowen}.

We consider the generalized $q$-dimension of the image measure of Gibbs measures supported on $\Sigma^\infty$. Given an  ACIFS, the pressure function $P^\mu(t, q)$ exists, see Lemma \ref{lem_Adq}, and $\underline{d}_q=\overline{d}_q=d_q$.
\begin{thm}\label{thm_HNCIFS}
Given an  ACIFS $\boldsymbol{\Phi}$ satisfying the open set condition. Let $\mu$ be a Gibbs measure on $\Sigma^\infty$ and $\mu^\omega$ be the image measure of $\mu$ defined by \eqref{mu} satisfying bounded overlap condition. Then we have 
$$
{D}_q(\mu^\omega)=d_q,
$$
for $q>0, q\neq 1$. 
\end{thm}

Let $\mu$ be the Bernoulli measure on $\Sigma^\infty$ defined by \eqref{b m} with positive probability vector $(p_1, \ldots, p_n)$. Since $\mu$ is a Gibbs measure on $\Sigma^\infty$,   the following is an immediate consequence of Theorem \ref{thm_HNCIFS}.

\begin{cor}
Given an ACIFS $\boldsymbol{\Phi}$ satisfying open set condition.  Let $\mu$ be a Bernoulli measure on $\Sigma^\infty$ defined by \eqref{b m}. Suppose the image measure  $\mu^\omega$  of $\mu$ satisfies  bounded overlap condition. Then for $q>0, q\not=1$,  we have
$$
{D}_q(\mu^\omega)=d_q.
$$
\end{cor}

The paper is organised as follows. In Section \ref{sec:GRFBD}, we give the properties of pressure functions and bounded distortion.
In Section \ref{sec:GQBM},  we study the generalized $q$-dimensions of Borel meausre and give the proofs of Theorem \ref{upper bound},  Theorem \ref{Dq E} and  Theorem \ref{simple}. In the final section, we study the generalized $q$-dimensions of Gibbs meausres for ACIFS and give the proof of Theorem \ref{thm_HNCIFS}.

\section{Generalized Pressure functions and bounded distortion}\label{sec:GRFBD}
In this section, we study the properties of pressure functions and bounded distortion which  are  useful to explore  the  non-autonomous conformal fractals.
From now on, we always  write $\boldsymbol{\Phi}$ for the non-autonomous conformal iterated function system.

Throughout, we use $C$ for a constant and write $Y \lesssim_t X$ to mean $Y \le C X$ for some constant $C > 0$ depending on $t$; similarly, $Y \gtrsim_t X$ means $Y \ge C X$ for some $C> 0$ depending on $t$. We write $Y \asymp_t X$ if both $Y \lesssim_t X$ and $Y \gtrsim_t X$ hold.

The following conclusion is straightforward. 
\begin{lem}\label{u contraction}
Given a $NCIFS$, for each integer $k$ and $u\in I_k$, 
$$
\|D\varphi_u\|\le c,
$$
where $c$ is given by \eqref{def_uccdn}.
\end{lem}

The following properties of generalized pressure functions are essential in the study of nonautonomous conformal sets.
\begin{lem}\label{m d}
Both  $\overline P_\mu(t, q)$ and  $\underline P_\mu(t, q)$  in \eqref{def_GULPF} are  monotonically  decreasing  in $t$. 

In particular, given $q>0$,   if $\overline P_\mu(t, q)$ and  $\underline P_\mu(t, q)$ is finite on an  interval $I$, then they are strictly decreasing on $I$,  convex  when $0<q<1$ and  concave   when $q>1$. Moreover $\overline{d}_q^*$ and $\underline{d}_q^*$ in \eqref{d*} are finite. 
\end{lem}

\begin{proof}
Given $q\in(0, 1)$,  suppose $\overline P_\mu(t_1, q)$ and $\overline P_\mu(t_2, q)$ are finite for $t_2>t_1$. It follows by Lemma \ref{u contraction} that
$$
\sum_{\mathbf{u}\in\mathcal{C}(\delta)}\|D\varphi_\mathbf{u}\|^{t_2(1-q)}\mu([\bu])^q
\le c^{k_\delta(1-q)(t_2-t_1)}\sum_{\mathbf{u}\in\mathcal{C}(\delta)}\|D\varphi_\mathbf{u}\|^{t_1(1-q)}\mu([\bu])^q. 
$$
Since $0<c<1$, by \eqref{def_GULPF},  we have that
$$
\overline P_\mu(t_2, q)\le\overline P_\mu(t_1, q),   \qquad  \underline P_\mu(t_2, q)\le\underline P_\mu(t_1, q).
$$
For $q=1$ and $q>1$,  by the similar argument,  we have  $\overline P_\mu(t_2, q) \leq  \overline P_\mu(t_1, q)$ and $\underline P_\mu(t_2, q)\le\underline P_\mu(t_1, q)$.

For $q>0, q\not=1$ and $0<\alpha<1$, by Hölder inequality, we have
\begin{eqnarray*}
&&\sum_{\mathbf{u}\in\mathcal{C}(\delta)}\|D\varphi_\mathbf{u}\|^{(\alpha t_1+(1-\alpha) t_2)(1-q)}\mu([\bu])^q   \\
&=&\sum_{\mathbf{u}\in\mathcal{C}(\delta)}(\|D\varphi_\mathbf{u}\|^{t_1(1-q)}\mu([\bu])^q)^\alpha(\|D\varphi_\mathbf{u}\|^{t_2(1-q)}\mu([\bu])^q)^{1-\alpha}  \\
&\le&\Big(\sum_{\mathbf{u}\in\mathcal{C}(\delta)}\|D\varphi_\mathbf{u}\|^{ t_1(1-q)}\mu([\bu])^q\Big)^\alpha \Big(\sum_{\mathbf{u}\in\mathcal{C}(\delta)}\|D\varphi_\mathbf{u}\|^{ t_1(1-q)}\mu([\bu])^q \Big)^{1-\alpha},
\end{eqnarray*}
which implies that 
\begin{equation*}\begin{split}
&\overline P_\mu(\alpha t_1+(1-\alpha) t_2, q)\le\alpha \overline P_\mu(t_1, q)+(1-\alpha)\overline P_\mu( t_2, q), \quad \textit{ for $0<q<1$} \\ 
&\overline P_\mu(\alpha t_1+(1-\alpha) t_2, q)\ge\alpha \overline P_\mu(t_1, q)+(1-\alpha)\overline P_\mu( t_2, q) ,  \quad \textit{ for $q>1$}. 
\end{split}
\end{equation*}
  Hence both $\overline P_\mu(t, q)$ and  $\underline P_\mu(t, q)$  are   convex  for $0<q<1$ and  concave   for $q>1$.
\end{proof}

Next, to study the   principle of bounded distortion, we cite the following mean value theorem;  see \cite{Rud76}. 
\begin{lem}[Quasi-differential Mean Value Theorem]\label{quasi_diff_mvt}
Let $\Omega \subseteq \mathbb{R}^d$ be an open convex set, and let $f: \Omega \to \mathbb{R}^d$ be a differentiable mapping. Then for all distinct points $x, y \in \Omega$, there exist a point $\xi$ on the line segment connecting $x$ and $y$ such that
$$
|f(x) - f(y)|\le \|Df(\xi)\||x - y|.
$$
\end{lem}
The principle of bounded distortion makes precise the idea of a set being 
'approximately nonautonomous similar', in that any sufficiently small neighbourhood may 
be mapped onto a large part of the set by a transformation that is not unduly 
distorting. 

\begin{lem}\label{contraction}
Given NCIFS $\boldsymbol{\Phi}$, for all $\mathbf{u}\in\Sigma^*$,  we have that for all $x, y\in J$,
\begin{equation}\label{ineq_BDxyL}
\|D\varphi_\mathbf{u}\||x-y|\asymp|\varphi_\mathbf{u}(x)-\varphi_\mathbf{u}(y)|, 
\end{equation}
and moreover,
\begin{equation}
\|D\varphi_\mathbf{u}\|\asymp|J_\mathbf{u}|. 
\end{equation}
\end{lem}	

\begin{proof}
Recall that $J$ is compact and $V$ is an open connected set containing $J$. The collection $\mathcal{F}=\{B(x, r_x)\subset V:x\in J\}$   of  balls is a cover of  $J$. Thus, there exists  a finite subcover $\{B(x_i, r_i)\}_{i=1}^k\subset\mathcal{F}$ of $J$. Let $\delta$ be the Lebesgue number of $\{B(x_i, r_i)\}_{i=1}^k$;  see \cite{PW}. Since $V$ is a connected set of $  \mathbb{R}^d$, it is also path-connected. For every $1\le i\le k-1$, there exists a path connecting $x_i$ and $x_{i+1}$, and we choose a finite number of balls $\{B(z_j, r_j)\}_{j=1}^{K_i}$ with  $z_1=x_i$ and $z_{K_i}=x_{i+1}$ satisfying  $B(z_j, r_j)\cap B(z_{j+1}, r_{j+1})\not=\emptyset$. We denote the new collection of these balls by $\{B(z_j, r_j)\}_{j=1}^l$.  Note that for each $j $, $B(z_j, r_j)\cap B(z_{j+1}, r_{j+1})\not=\emptyset$.  

Fix $\bu\in\Sigma^*$. Arbitrarily choose  $x, y\in J$. There exist integers $m$ and  $n$ such that $x\in B(z_m, r_m)$ and $y\in B(z_n, r_n)$.

We first  show  
\begin{equation}\label{ineq_BDxyR}
|\varphi_\mathbf{u}(x)-\varphi_\mathbf{u}(y)|\lesssim \|D\varphi_\mathbf{u}\||x-y|.
\end{equation}
If $0<|x-y|<\delta$,  then by Lemma \ref{quasi_diff_mvt}, it holds.  
Otherwise, for  $|x-y|\geq \delta$,  by Lemma \ref{quasi_diff_mvt}, we have
\begin{eqnarray*}
|\varphi_\mathbf{u}(x)-\varphi_\mathbf{u}(y)|&\le&|\varphi_\mathbf{u}(x)-\varphi_\mathbf{u}(z_m)|+\cdots+|\varphi_\mathbf{u}(z_n)-\varphi_\mathbf{u}(y)|   \\
&\le&\|D\varphi_\mathbf{u}\||x-z_m|+\cdots+\|D\varphi_\mathbf{u}\|z_n-y|  \\
&\le& 2\delta(\|D\varphi_\mathbf{u}\|+\cdots+2\|D\varphi_\mathbf{u}\|) \\
&\lesssim&\|D\varphi_\mathbf{u}\||x-y|.
\end{eqnarray*}
Hence  \eqref{ineq_BDxyR} holds.

Next, we show  \begin{equation}\label{ineq_BDxyL}
|\varphi_\mathbf{u}(x)-\varphi_\mathbf{u}(y)|\gtrsim \|D\varphi_\mathbf{u}\||x-y|.
\end{equation} 
Suppose $\varphi_\bu(x)\in B(z_p, r_p)$ and $\varphi_\bu(y)\in B(z_q, r_q)$ where $1\le p\le q\le l$. 
If $p = q$, let $L$ be the line segment connecting $\varphi_\bu(x)$ and $\varphi_\bu(y)$. If $p < q$, let $L$ be the polyline connecting the points in the following order:
$\varphi_\bu(x)$, $z_p$, $z_{p+1}$, $\dots$, $z_q$, $\varphi_\bu(y)$. We may parameterize $L$ by a continuous map $L:[0,1]\to\mathbb{R}^d$ with $L(0)=\varphi_\bu(x)$ 
and $L(1)=\varphi_\bu(y)$. For $s\in[0,1]$, define $L_s = \{L(t):0\le t\le s\}$ and let 
$|L_s|$ denote the length of $L_s$.  Note that $|L|\leq 2\sum_{j=1}^lr_j$

Let $t_0=\sup\{t\in[0,1] :L_t\subset\varphi_\bu(V)\}$. Then
$$
|L|\ge   |L_{t_0}|\gtrsim \|D\varphi_\bu\|\it{dist}(\partial V, J)\frac{|x-y|}{|J|}.     
$$
Since $\varphi_\bu(x), \varphi_\bu(y)\in J$, if $|\varphi_\bu(x)-\varphi_\bu(y)|\le\delta$, then $|\varphi_\bu(x)-\varphi_\bu(y)|=|L|$; if $|\varphi_\bu(x)-\varphi_\bu(y)|>\delta$, then   
$$
\frac{2\sum_{j=1}^lr_j}{\delta}|\varphi_\bu(x)-\varphi_\bu(y)|\ge2\sum_{j=1}^lr_j\ge|L|.
$$              
Thus
$$
|\varphi_\bu(x)-\varphi_\bu(y)|\ge\frac{\delta \it{dist}(\partial V, J)}{2C\sum_{j=1}^lr_j|J|}\|D\varphi_\mathbf{u}\||x-y|.
$$

Since $\Psi_\bu(x)=\varphi_\bu(x)+\omega_\bu$, it follows that
$$
|\Psi_\mathbf{u}(x)-\Psi_\mathbf{u}(y)|=|\varphi_\mathbf{u}(x)-\varphi_\mathbf{u}(y)|,
$$
and we have $\|D\varphi_\mathbf{u}\||x-y|\asymp|\varphi_\mathbf{u}(x)-\varphi_\mathbf{u}(y)|$.
\end{proof}

\begin{lem}\label{cor_subMul}
Given NCIFS $\boldsymbol{\Phi}$, and given integers $0< m< n<\infty$, for all  $\bu\in\Sigma^m$ and $\bv\in\Sigma_{m+1}^n$,  we have that
\begin{equation}
\|D\varphi_{\bu}\|\|D\varphi_{\bv}\|\asymp \|D\varphi_{\bu\bv}\|.
\end{equation}
\end{lem}	

\begin{proof}
Since $D\varphi_{\mathbf{u}\bv}(x)=D\varphi_{\mathbf{u}}(\varphi_{\bv}(x))D\varphi_{\bv}(x)$, it follows that
$$
\|D\varphi_{\bu\bv}\|\le \|D\varphi_{\bu}\|\|D\varphi_{\bv}\|.
$$
By (v) in the definition, for every $x\in J$, we have that 
\begin{eqnarray*}
\|D\varphi_\mathbf{uv}\|&\ge& \|D\varphi_{\mathbf{u}}(\varphi_{\bv}(x))\|\|D\varphi_{\bv}(x)\|  \ge C^{-1}\|D\varphi_{\mathbf{u}}\|\|D\varphi_{\bv}(x)\|,
\end{eqnarray*}
and the conclusion holds.
\end{proof}

\begin{lem}\label{vol}
Given NCIFS $\boldsymbol{\Phi}$, for all $\mathbf{u}\in\Sigma^*$ and $A\subset V$, we have
\begin{equation}
\|D\varphi_\mathbf{u}\|^d\mathcal{L}^d(A)\asymp\mathcal{L}^d(\varphi_\mathbf{u}(A)).
\end{equation}
\end{lem}	

\begin{proof}
Since $\varphi_\bu$ is a $C^1$ conformal diffeomorphism, it is clear that
$$
\mathcal{L}^d(\Psi_\mathbf{u}(A))=\int_{A}\|D\varphi_\mathbf{u}(x)\|^dd\mathcal{L}^d.
$$
By the bounded distortion property, we have
$$
C^{-d}\|D\varphi_\mathbf{u}\|^d\mathcal{L}^d(A)\le\mathcal{L}^d(\Psi_\mathbf{u}(A))\le \|D\varphi_\mathbf{u}\|^d\mathcal{L}^d(A),
$$
and the conclusion holds.
\end{proof}

\section{Generalized $q$-dimension of Borel measures}\label{sec:GQBM}
In this section, we present the formula for the 
$L^q$-spectrum of positive finite Borel measures supported on $\Sigma^\infty$.

Recall that for $0<\delta<M_1$, we write $\mathcal{C}(\delta)=\{\mathbf{u}\in\Sigma^*:   \|D\varphi_\mathbf{u}\|\le \delta< \|D\varphi_{\mathbf{u}^*}\|\}$.

\begin{lem}\label{u<Q}
Let $\mu$ be a finite Borel measure on $\Sigma^\infty$, let $\mu^\omega$ be defined by \eqref{mu}. Then for all $\omega$ and for all sufficiently small $\delta>0$
\begin{equation*}\begin{split}
&\sum_{\mathbf{u}\in\mathcal{C}(\delta)}\mu([\bu])^q\gtrsim_q\sum_{Q\in{\mathcal{M}_\delta}}\mu^\omega(Q)^q, \qquad\qquad\qquad\qquad\qquad \textit{ for $0<q<1$} \\ 
&\sum_{\mathbf{u}\in\mathcal{C}(\delta)}\mu([\bu])\log \mu([\bu])-\sum_{Q\in{\mathcal{M}_\delta}}\mu^\omega(Q)\log \mu^\omega(Q)\lesssim 1, \quad \textit{ for $q=1$} \\ 
&\sum_{\mathbf{u}\in\mathcal{C}(\delta)}\mu([\bu])^q\lesssim_q\sum_{Q\in{\mathcal{M}_\delta}}\mu^\omega(Q)^q,  \qquad\qquad\qquad\qquad\qquad \textit{ for $q>1$}. 
\end{split}
\end{equation*}
\end{lem}

\begin{proof}
Given $\bu\in\Sigma^*$, let $\mu_\mathbf{u}$ denote the restriction of $\mu$ to the cylinder $[\bu]$, and let $\mu^\omega_\mathbf{u}$ be the image measure of $\mu_\mathbf{u}$ under $\pi_{\boldsymbol{\Phi}}$. It is clear that the support of $\mu^\omega_\mathbf{u}$ is contained in $J_\mathbf{u}$, that is, spt $\mu^\omega_\mathbf{u}\subset J_\mathbf{u}$, and
$$
\mu([\bu])=\mu_\mathbf{u}([\bu])=\mu^\omega_\mathbf{u}(J_\mathbf{u}).
$$
By Lemma \ref{contraction}, there exists a constant $C_1$ such that for each $\delta<\min\{\|D\varphi_u\|:u\in I_1\}$ and every $\mathbf{u}\in\mathcal{C}(\delta)$ we have $|J_\bu|\le C_1\|D\varphi_\bu\|$, and it follows that there exists a constant $C_2$ such that $J_\mathbf{u}$ intersects at most $C_2 3^d$  $\delta$-cubes in $\mathbb{R}^d$.

For $0<q<1$,  by Jensen's inequality, we have that  for each $\mathbf{u}\in\mathcal{C}(\delta)$,
$$
\mu([\bu])^q=\mu^\omega_\mathbf{u}(J_\mathbf{u})^q\ge (C_23^d)^{(q-1)}\sum_{Q\in{\mathcal{M}_\delta}}\mu^\omega_\mathbf{u}(Q)^q.
$$
For each $Q\in\mathcal{M}_\delta$, by the power inequality, we have
$$
\sum_{\mathbf{u}\in\mathcal{C}(\delta)}\mu^\omega_\mathbf{u}(Q)^q\ge\big(\sum_{\mathbf{u}\in\mathcal{C}(\delta)}\mu^\omega_\mathbf{u}(Q)\big)^q=\mu^\omega(Q)^q.
$$
It follows that
$$
\sum_{\mathbf{u}\in\mathcal{C}(\delta)}\mu([\bu])^q  \gtrsim_q \sum_{Q\in{\mathcal{M}_\delta}}\mu^\omega(Q)^q.
$$
The other proofs are similar, and  the conclusion holds.
\end{proof}

\begin{lem}\label{condition app}
Given  $\boldsymbol{\Phi}$  satisfying \eqref{condition}, and given $t<t'<t+1$, there exists  $\Delta>0$ such that for every $\delta<\Delta$ and all  $\bu\in\mathcal{C}(\delta)$, 
\begin{equation*}\begin{split}
&\delta^{t'(1-q)}\le\|D\varphi_\bu\|^{t(1-q)}, \quad \ \textit{ for $0<q<1$} \\ 
&\delta^{t'}\le\|D\varphi_\bu\|^{t}, \qquad\qquad\quad \textit{ for $q=1$} \\ 
&\delta^{t'(1-q)}\ge\|D\varphi_\bu\|^{t(1-q)},  \quad \ \textit{ for $q>1$}. 
\end{split}
\end{equation*}
\end{lem}

\begin{proof}
For $q>1$, if $t'\leq 0$, by Lemma \ref{u contraction}, for $\delta<|J|$ and   $\bu\in\mathcal{C}(\delta)$,  
$$
\delta^{t'(1-q)}\ge \|D\varphi_\bu\|^{t'(1-q)}\ge c^{(t'-t)(1-q)}\|D\varphi_\bu\|^{t(1-q)}\ge \|D\varphi_\bu\|^{t(1-q)}.
$$
If $t'>0$, by Lemma \ref{cor_subMul}, for each  $\bu\in\mathcal{C}(\delta)$,  we have that
$$
\delta^{t'(1-q)}\ge(\|D\varphi_{\mathbf{u}^*}\|)^{t'(1-q)}   \ge \Big(\frac{\|D\varphi_\mathbf{u}\|}{\underline c_{|\mathbf{u}|}}\Big)^{t'(1-q)}, 
$$
where $\underline c_{|\mathbf{u}|}$ is given by \eqref{Mk}. Since $\lim_{k\to +\infty} \frac{\log\underline c_k}{\log M_k}=0,$ there exists $K>0$ such that for $k>K$, 
$$
\frac{ M_k^{t'-t}}{\underline c_k^{t'}}<1.
$$
Let  $\Delta=M_K$. For $\delta<\Delta$, we have that 
\begin{eqnarray*}
\delta^{t'(1-q)}&\ge&\Big(\frac{\|D\varphi_\mathbf{u}\|^tM_{|\bu|}^{t'-t}}{\underline c_{|\mathbf{u}|}^{t'}}\Big)^{1-q}    \ge \|D\varphi_\mathbf{u}\|^{t(1-q)}.
\end{eqnarray*}
  
The proofs for $0<q\le1$ are similar, we omit them.
\end{proof}

For  $F\subset \mathbb{R}^d$ such that  $E\cap F\ne \emptyset$, we write
\begin{equation}\label{def_Aset}
A(F)=\{\mathbf{u}\in \mathcal{C}(|F|): J_\mathbf{u}\cap F \ne \emptyset\}.
\end{equation}
Let
\begin{equation}\label{def_k0}
k_F^- = \min\{k:|\mathbf{u}|=k, \mathbf{u}\in A(F)\}, \qquad
k_F^+ = \max\{k:|\mathbf{u}|=k, \mathbf{u}\in A(F)\}.
\end{equation}
For each integer $k_F^-\le k\leq k_F^+$, we write
\begin{equation}\label{def_DFk}
D(F, k)=\{\mathbf{u} \in \Sigma^k :\mathbf{u} \in A(F)\}.
\end{equation}

\begin{lem}\label{finite intersection}
Let $\boldsymbol{\Phi}$ be a non-autonomous conformal iterated function system satisfying open set condition. Then for every  $F \subset \mathbb{R}^d$ with  $E\cap F\ne \emptyset$, we have
$$
\sum_{k = k_F^-}^{k_F^+}\underline c_k^d\# D(F, k)\lesssim 1,
$$
where $k_F^-$ and $k_F^+$ are given by \eqref{def_k0}.
\end{lem}

\begin{proof}
Given a set $F \subset \mathbb{R}^d$ such that   $E\cap F\ne \emptyset$. Let $\delta=|F|$. For   $\bu=u_1u_2\ldots u_k\in \mathcal{C}(\delta) $, Since $\underline c_k=\min_{1\le j \le \# I_k} \{\|D\varphi_{k, j}\|\}$,  it is clear that for all $x\in J$ 
$$
\|D\varphi_\mathbf{u}\|\ge  \|D\varphi_{\mathbf{u}^*}(\varphi_{u_k}(x))D\varphi_{u_k}(x)\|  \gtrsim\|D\varphi_{\mathbf{u}^*}\|\underline c_k.
$$
Since $\|D\varphi_{\mathbf{u}^*}\|> \delta$, it implies that 
\begin{equation}\label{DckD}
\|D\varphi_\mathbf{u}\|\ge \underline c_k\delta.
\end{equation} 
Hence it follows that 
$$
\delta ^d\sum_{k = k_F^-}^{k_F^+}\underline c_k^d\# D(F, k) \le\sum_{k = k_F^-}^{k_F^+}\sum_{\mathbf{u} \in D(F, k)} \|D\varphi_\mathbf{u}\|^d=\sum_{\mathbf{u} \in A(F)} \|D\varphi_\mathbf{u}\|^d  . 
$$

Fix $x\in F$. By Lemma \ref{contraction}, we have $|J_\bu|\lesssim\|D\varphi_\bu\|<\delta$, and there exists a constant $C_1$ independent of $x$ such that $J_\mathbf{u} \subset B(x,2C_1\delta)$ for all $ \mathbf{u} \in A(F)$,   by Lemma \ref{contraction} and Lemma \ref{vol}, 
$$
 \mathcal{L}^d(\mbox{int}(J))\delta ^d\sum_{k = k_F^-}^{k_F^+}\underline c_k^d\# D(F, k)\le \mathcal{L}^d(\mbox{int}(J)) \sum_{\mathbf{u} \in A(F)} \|D\varphi_\mathbf{u}\|^d  \lesssim   \mathcal{L}^d(B(x,2C_1\delta).
$$
Since  $\mathcal{L}^d(B(x,2C_1\delta))\asymp \delta^d  \mathcal{L}^d(B(0,1)), $   the conclusion holds.
\end{proof}

\begin{proof}[Proof of Theorem \ref{upper bound}]
We only give the proof for $\underline{D}_q(\mu^\omega)\le\min\{\underline{d}_q^*, d\}$ since the other is similar. 
It is sufficient to prove that for all $\underline{d}_q^*<t<d$, we have
$$
\underline{D}_q(\mu^\omega)\le t,
$$
where $\underline{d}_q^*$ is given by \eqref{d*}.

For $q>1$, since $\underline{d}_q^*<t$, by Lemma \ref{m d},  we have $\underline P_\mu(t, q)<0$, and by \eqref{def_GULPF}, there exists $\{\delta_k\}$ such that
$$
\sum_{\mathbf{u}\in\mathcal{C}(\delta_k)}\|D\varphi_\mathbf{u}\|^{t(1-q)}\mu([\bu])^q>e^{-\frac{1}{2}k_{\delta_k} \underline P_\mu(t, q)}.
$$
 For each  $t+1>t'>t$, since $\lim_{k\to +\infty} \frac{\log\underline c_k}{\log M_k}=0$, by Lemma \ref{condition app}, it is follows that for sufficiently large $k$ 
$$
\sum_{\mathbf{u}\in\mathcal{C}(\delta_k)}\delta_k^{t'(1-q)}\mu([\bu])^q\ge e^{-\frac{1}{2}k_{\delta_k}\underline P_\mu(t, q)}>1.
$$
By Lemma \ref{u<Q}, it implies that 
$$
t'>\liminf_{\delta\to 0}\frac{\log \sum_{\mathbf{u}\in\mathcal{C}(\delta)}\mu([\bu])^q}{(q-1)\log \delta}\ge\liminf_{\delta\to 0}\frac{\log \sum_{Q\in{\mathcal{M}_\delta}}\mu^\omega(Q)^q}{(q-1)\log \delta}=\underline{D}_q(\mu^\omega).
$$
Hence  $\underline{D}_q(\mu^\omega)\le t'$  for all   $t+1>t'>t$, and we obtain  $\underline{D}_q(\mu^\omega)\le \underline{d}_q^*.$

For $q=1$, similarly, we have  $\underline P_\mu(t, q)<0$  since $\underline{d}_q^*<t$,  and   there exists $\{\delta_k\}$ such that
and
$$
\sum_{\mathbf{u}\in\mathcal{C}(\delta_k)}\mu([\bu])\log(\|D\varphi_\mathbf{u}\|^{-t}\mu([\bu]))>-\frac{1}{2}k_{\delta_k}  \underline P_\mu(t, q).
$$
 For each  $t+1>t'>t$, since $\lim_{k\to +\infty} \frac{\log\underline c_k}{\log M_k}=0$, by Lemma \ref{condition app},  for sufficiently large $k$ 
$$
\sum_{\mathbf{u}\in\mathcal{C}(\delta_k)}\mu([\bu])\log(\delta_k^{-t'}\mu([\bu]))\ge\sum_{\mathbf{u}\in\mathcal{C}(\delta_k)}\mu([\bu])\log(\|D\varphi_\mathbf{u}\|^{-t}\mu([\bu]))>1,
$$
and by Lemma \ref{u<Q}, there exists a constant $C_{t,q}$ such that
\begin{eqnarray*}
t'>\liminf_{\delta\to 0}\frac{ \sum_{\mathbf{u}\in\mathcal{C}(\delta)}\mu([\bu])\log\mu([\bu])}{\log \delta} \ge\liminf_{\delta\to 0}\frac{ \sum_{Q\in{\mathcal{M}_\delta}}\mu^\omega(Q)\log\mu^\omega(Q)-C_{t,q}}{\log \delta}  =\underline{D}_1(\mu^\omega).
\end{eqnarray*}
Hence  $\underline{D}_1(\mu^\omega)\le t'$ for all $t+1>t'>t$, and we obtain  $\underline{D}_1(\mu^\omega)\le \underline{d}_1^*.$

The proof for $0<q<1$ is similar, we omit it.
\end{proof}

For each $\delta>0$, we write $\underline{c}_\delta=\min\{\underline c_{|\bu|}:\mathbf{u}\in \mathcal{C}(\delta)\}$.

\begin{lem}\label{u=Q}
Given  $\boldsymbol{\Phi}$ satisfying \eqref{condition}, let $\mu$ be a finite Borel measure on $\Sigma^\infty$ satisfying BOC, and let $\mu^\omega$ be defined by \eqref{mu}. Then for all $\omega$ and for all sufficiently small $\delta>0$
\begin{equation*}\begin{split}
&\underline{c}_\delta^{d^2}\sum_{Q\in{\mathcal{M}_\delta}}\mu^\omega(Q)^q\lesssim_q\sum_{\mathbf{u}\in\mathcal{C}(\delta)}\mu([\bu])^q\lesssim_q \underline{c}_\delta^{-d^2}\sum_{Q\in{\mathcal{M}_\delta}}\mu^\omega(Q)^q, \quad \textit{ for $0<q<1$} \\ 
&\sum_{\mathbf{u}\in\mathcal{C}(\delta)}\mu([\bu])\log \mu([\bu])-\sum_{Q\in{\mathcal{M}_\delta}}\mu^\omega(Q)\log \mu^\omega(Q)\asymp\log \underline{c}_\delta, \quad \textit{ for $q=1$} \\ 
&\underline{c}_\delta^{d^2q}\sum_{Q\in{\mathcal{M}_\delta}}\mu^\omega(Q)^q\lesssim_q\sum_{\mathbf{u}\in\mathcal{C}(\delta)}\mu([\bu])^q\lesssim_q \underline{c}_\delta^{-d^2q}\sum_{Q\in{\mathcal{M}_\delta}}\mu^\omega(Q)^q,  \quad \textit{ for $q>1$}. 
\end{split}
\end{equation*}
\end{lem}

\begin{proof}
For sufficiently small $\delta>0$, recall that $\mathcal{C}(\delta) =\{\mathbf{u}\in\Sigma^*: \|D\varphi_\mathbf{u}\|\le \delta< \|D\varphi_{\mathbf{u}^*}\| \}$ is a cut set, and there exists a constant $C_1$ such that for every $\mathbf{u}\in\mathcal{C}(\delta)$,  $J_\mathbf{u}$ intersects at most $C_13^d$  $\delta$-cubes in $\mathbb{R}^d$. Since $\lim_{k\to +\infty} \frac{\log\underline c_k}{\log M_k}=0$, by Lemma \ref{finite intersection}, each $\delta$-cube $Q$ intersects at most $\frac{C_q}{\underline{c}_\delta^d}$ basic sets in $\{J_\mathbf{u}:\mathbf{u}\in\mathcal{C}(\delta)\}$,  where $C_q$ is a constant.

For $q>1$. Recall that $\mathcal{M}_\delta$ is the family of $\delta$-mesh cubes in $\mathbb{R}^d$, and for each $Q\in\mathcal{M}_\delta$, we have
$$
\mu^\omega(Q)^q\le\big(\frac{C_q}{\underline{c}_\delta^d}\big)^{d(q-1)}\sum_{\mathbf{u}\in\mathcal{C}(\delta)}\mu^\omega(Q\cap J_\mathbf{u})^q.
$$
It follows that
$$
\sum_{Q\in{\mathcal{M}_\delta}}\mu^\omega(Q)^q \lesssim_q \underline{c}_\delta^{-d^2q} \sum_{\mathbf{u}\in\mathcal{C}(\delta)}\mu^\omega( J_\mathbf{u})^q.
$$
Since 
$$
\mu^\omega(J_\bu)\asymp\mu([\bu]),
$$ 
combining with Lemma \ref{u<Q}, we have that 
$$
\underline{c}_\delta^{d^2q}\sum_{Q\in{\mathcal{M}_\delta}}\mu^\omega(Q)^q\lesssim_q\sum_{\mathbf{u}\in\mathcal{C}(\delta)}\mu([\bu])^q\lesssim_q \underline{c}_\delta^{-d^2q}\sum_{Q\in{\mathcal{M}_\delta}}\mu^\omega(Q)^q. 
$$

The other proofs are similar, and we omit it.

\end{proof}

\begin{proof}[Proof of Theorem \ref{Dq E}]
We only give the proof for $\underline{D}_q(\mu^\omega)=\underline{d}_q^*$. By Theorem \ref{upper bound}, it suffices to prove $\underline{D}_q(\mu^\omega)\ge t$ for all $t<\underline{d}_q^*$.

For $q>1$, since $\underline P_\mu(t, q)>0$, by \eqref{def_GULPF}, there exists $\Delta>0$ such that
for all $\delta>\Delta$, 
$$
\sum_{\mathbf{u}\in\mathcal{C}(\delta)}\|D\varphi_\mathbf{u}\|^{t(1-q)}\mu([\bu])^q<e^\frac{-k_{\delta}  \underline P_\mu(t, q)}{2}.
$$
For $t+1>t'>t$,  we have that for all $\bu\in\mathcal{C}(\delta)$, 
$$
\delta^{t'(1-q)}\le\|D\varphi_\bu\|^{t(1-q)},
$$
and it follows that for sufficiently large $k_\delta$, 
$$
\sum_{\mathbf{u}\in\mathcal{C}(\delta)}\delta^{t'(1-q)}\mu([\bu])^q\le e^\frac{-k_{\delta}  \underline P_\mu(t, q)}{2}<1.
$$
By Lemma \ref{u=Q},     there exist constants $C_{t,q}, C_{t,q}'$ such that
$$
t'<\frac{\log\sum_{\mathbf{u}\in\mathcal{C}(\delta)}\mu([\bu])^q- C_{t,q}}{(q-1)\log \delta}<\frac{ C_{t,q}'\log \underline c_\delta+\log\sum_{Q\in{\mathcal{M}_\delta}}\mu^\omega(Q)^q}{(q-1)\log \delta}.
$$

For each $\delta>0$, there exists  $k>0$ such that $\underline c_\delta=\underline c_k$, and it follows that 
$$
0\le\frac{\log\underline c_\delta}{\log \delta}=\frac{\log\underline c_k}{\log \delta}\le      \frac{\log\underline c_k}{\log C+\log M_k-\log \underline c_k}. 
$$
Since $\lim_{k\to +\infty} \frac{\log\underline c_k}{\log M_k}=0,$ we have
$$
\lim_{\delta\to 0} \frac{\log \underline c_\delta}{\log \delta}=0,
$$
and it implies 
$$
t'\le\liminf_{\delta\to 0}\frac{\log \sum_{Q\in{\mathcal{M}_\delta}}\mu^\omega(Q)^q}{(q-1)\log \delta}=\underline{D}_q(\mu^\omega), 
$$
for  all $t+1>t'>t$,
Hence  $\underline{D}_q(\mu^\omega)\ge t$.

The proofs for $0<q<1$ and $q=1$ are similar, and  we omit them.
\end{proof}

The following conclusion follows by the same argument of Lemma \ref{m d}. 
\begin{cor}\label{m d1}
Both  $\overline P^\mu(t, q)$ and  $\underline P^\mu(t, q)$ given by \eqref{def_ULPF} are  monotonously decreasing  in $t$. In particular, given $q>0$,   if $\overline P^\mu(t, q)$ and  $\underline P^\mu(t, q)$ are finite on an  interval $I$, then they are strictly decreasing on $I$,  convex  when $0<q<1$ and  concave   when $q>1$. Moreover, $\overline{d}_q$ and $\underline{d}_q$ in \eqref{dl} are finite.  
\end{cor}

\begin{prop}\label{cut to qici}
Given NCIFS $\boldsymbol{\Phi}$.  Let $\mu$ be a positive finite Borel  measure on $\Sigma^\infty$, and let $\mu^\omega$ be the image measure of $\mu$. Then the following numbers  are all equal:
\[\begin{split}
&\overline{d}_q= \overline{d}_q^*=\inf\{t:\sum_{k=1}^\infty\sum_{\mathbf{u}\in\Sigma^k}\|D\varphi_\mathbf{u}\|^{t(1-q)}\mu([\bu])^q<\infty\} \qquad \textit{ for $0<q<1$} ; \\
&\underline{d}_q= \underline{d}_q^*=\sup\{t:\sum_{k=1}^\infty\sum_{\mathbf{u}\in\Sigma^k}\|D\varphi_\mathbf{u}\|^{t(1-q)}\mu([\bu])^q<\infty\} \qquad \textit{ for $q>1$}.
\end{split}
\]
\end{prop}

\begin{proof}
Since the proofs are similar, we only give the one for $q>1$. We write 
$$
d_q^1=\sup\{t:\sum_{k=1}^\infty\sum_{\mathbf{u}\in\Sigma^k}\|D\varphi_\mathbf{u}\|^{t(1-q)}\mu([\bu])^q<\infty\}.
$$
Note that $\overline{d}_q, \overline{d}_q^*$ and $d_q^1$ may take values of  $\infty$ and $-\infty$.

First, we show $\underline{d}_q\leq d_q^1 $.  For each non-integral $t$ such that $t<\underline{d}_q $, by \eqref{def_ULPF} and \eqref{dl}, there exists  $K_1>0$ such that
$$
\sum_{\mathbf{u}\in\Sigma^k}\|D\varphi_\mathbf{u}\|^{t(1-q)}\mu([\bu])^q<e^{-\frac{1}{2}k\underline{P}_\mu(t,q)}
$$
for each $k>K_1$, and it follows that 
$$
\sum_{k=K_1}^\infty\sum_{\mathbf{u}\in\Sigma^k}\|D\varphi_\mathbf{u}\|^{t(1-q)}\mu([\bu])^q<\sum_{k=K_1}^\infty e^{-\frac{1}{2}k\underline{P}_\mu((t,q)}<\infty.
$$
Hence  $t<d_q^1$ for all $t<\underline{d}_q $, and we obtain $\underline{d}_q\leq d_q^1 $.

Next, we show $d_q^1\leq \underline{d}_q^*$. For each $t<d_q^1$ we have
$$
\sum_{k=1}^\infty\sum_{\mathbf{u}\in\Sigma^k}\|D\varphi_\mathbf{u}\|^{t(1-q)}\mu([\bu])^q<\infty,
$$
and for every cut set $\mathcal{C}$, we have
$$
\sum_{\mathbf{u}\in\mathcal{C}}\|D\varphi_\mathbf{u}\|^{t(1-q)}\mu([\bu])^q<\sum_{k=1}^\infty\sum_{\mathbf{u}\in\Sigma^k}\|D\varphi_\mathbf{u}\|^{t(1-q)}\mu([\bu])^q<\infty.
$$
Then we have $\underline P_\mu(t, q)\ge 0$, which means $t\le \underline{d}_q^*$ and $d_q^1\leq \underline{d}_q^*$.

Finally, we show $\underline{d}_q^*\leq \underline{d}_q$. For each $t<\underline{d}_q^*$,  we have $\underline P_\mu(t, q)> 0$, and  there exists a $\Delta>0$ such that for each $\delta<\Delta$
$$
\sum_{\mathbf{u}\in\mathcal{C}(\delta)}\|D\varphi_\mathbf{u}\|^{t(1-q)}\mu([\bu])^q<e^{k_\delta\frac{-\underline P_\mu(t, q)}{2}}<1.
$$
Choosing $\rho$ such that $c<\rho<1$ where $c$ is given by \eqref{def_uccdn}, there exist  an integer $K>0$ such that  $\rho^{K+1}<\Delta\le\rho^{K}$. Since  $\Sigma^*=\bigcup_{k=0}^\infty\Sigma^k=\bigcup_{k=0}^\infty\mathcal{C}(\rho^k)$,  by Lemma \ref{contraction}, it follows that for all $t'<t$, 
\begin{eqnarray*}
\sum_{\mathbf{u}\in\Sigma^k}\|D\varphi_\mathbf{u}\|^{t'(1-q)}\mu([\bu])^q&\le&\sum_{k=1}^\infty\sum_{\mathbf{u}\in\Sigma^k}\|D\varphi_\mathbf{u}\|^{t'(1-q)}\mu([\bu])^q  \\
&\le&\sum_{k=1}^\infty \sum_{\mathbf{u}\in\mathcal{C}(\rho^k)}(\|D\varphi_\mathbf{u}\|^{t} \|D\varphi_\mathbf{u}\|^{t'-t})^{1-q}\mu([\bu])^q  \\
&\le&\sum_{k=1}^{K} \sum_{\mathbf{u}\in\mathcal{C}(\rho^k)}\|D\varphi_\mathbf{u}\|^{t'(1-q)}\mu([\bu])^q +\sum_{k=K}^\infty (\rho^k)^{(t'-t)(1-q)}  \\
&<&\infty,
\end{eqnarray*}
Hence $\underline P^\mu(t', q)\leq  0$, and  $t'\le  \underline{d}_q$ for all $t'<t$. It follows that $\underline{d}_q^*\le \underline{d}_q$. 
\end{proof}

\begin{proof}[Proof of Theorem \ref{simple}]
By Theorem \ref{Dq E} and Proposition \ref{cut to qici}, we have $\underline{D}_q(\mu^\omega)=\underline{d}_q$ for $q>1$, and $
\overline{D}_q(\mu^\omega)=\overline{d}_q$ for $0<q<1$.
\end{proof}

\section{generalized $q$-dimensions of Gibbs measures}
In this section, we study  autonomous conformal iterated function systems, that is,  $\boldsymbol{\Phi}=\{\Phi_k\}_{k=1}^\infty$ with  $\Phi_k=\Phi_1$ for all $k\geq 1$. 

First, we show that the pressure function exists and is given by 
$$
P^\mu(t, q)= \lim_{k\to\infty}\frac{\log  \sum_{\mathbf{u}\in\Sigma^k}||D\varphi_\mathbf{u}||^{t(1-q)}\mu([\bu])^q}{k},
$$
and we have $ d_q=\inf\{t:  P^\mu(t, q)<0\}=\sup\{t: P^\mu(t, q)>0\}.$

\begin{lem}\label{lem_Adq}
Given an ACIFS. Let $\mu$ be a Gibbs measure on $\Sigma^\infty$. Then for all $t\in \mathbb{R}$ and $q>0, q\not=1$,  $P^\mu(t, q)$ exists, and moreover $\underline{d}_q=\overline{d}_q=d_q$ is the unique  solution to $P^\mu(d_q, q)=0.$
\end{lem}

\begin{proof}
Since  $S_{k+l}f(\mathbf{u})=S_{k}f(\mathbf{u})+S_{l}f(\sigma^k\mathbf{u}) $   for all $k, l\in\mathbb{N}_+$,  
 it follows from applying \eqref{Gibbs} to cylinders $[\bu], [\bv]$  and $[\bu\bv]$ that
$$
a^3\le\frac{\mu([\bu\bv])}{\mu([\bu])\mu([\bv])}\le\frac{1}{a^3}.
$$
Given $q>0, q\not=1$,  without loss of generality, we assume  $t>0$. Since 
$$
\sum_{\mathbf{u}\in\Sigma^{k+l}}||D\varphi_\mathbf{u}||^{t(1-q)}\mu([\bu])^q\lesssim_{t,q}  \sum_{\mathbf{u}\in\Sigma^k}||D\varphi_\mathbf{u}||^{t(1-q)}\mu([\bu])^q   \sum_{\mathbf{u}\in\Sigma^l}||D\varphi_\mathbf{u}||^{t(1-q)}\mu([\bu])^q, 
$$
it immediately follows that 
$$
P^\mu(t, q)= \lim_{k\to\infty}\frac{\log  \sum_{\mathbf{u}\in\Sigma^k}||D\varphi_\mathbf{u}||^{t(1-q)}\mu([\bu])^q}{k}
$$
exists.  
By Corollary \ref{m d1},  $P^\mu(t, q)$ is continuous in $t$ and  strictly monotonically decreasing with  $\lim_{t\to\infty}P^\mu(d_q, q)=-\infty$ and $\lim_{t\to-\infty}P^\mu(d_q, q)=\infty$. Hence, for each fixed $q$, there exists a unique $d_q$ such that $P^\mu(d_q, q)=0.$            
\end{proof}

\begin{prop}\label{prop_lme}
Let $\mu$ be a Gibbs measure on $\Sigma^\infty$. If  $0<q<1$ and $s>d_q$, or if $q>1$ and $0<s<d_q$, then
$$
\sum_{k=0}^\infty\sum_{\mathbf{u}\in \Sigma^k}||D\varphi_\mathbf{u}||^{s(1-q)}\mu([\mathbf{u}])^q<\infty.
$$
If $0<q<1$ and $0<s<d_q$, or if $q>1$ and $s>d_q$, then
$$
\lim_{k\to\infty}\min_{\mathcal{C}:k_\mathcal{C}\ge k}\sum_{\mathbf{u}\in\mathcal{C}}||D\varphi_\mathbf{u}||^{s(1-q)}\mu([\mathbf{u}])^q=\infty,
$$
where the minimum is over cut-set $\mathcal{C}$ for which $k_\mathcal{C}\ge k$.
\end{prop}

\begin{proof}
Recall that $M_k=\max_{\mathbf{u} \in \Sigma^k} \{\|D\varphi_\mathbf{u}\|\}$ for every $k\geq 1$. 

First consider  $q>1$ and  $0<s<d_q$.  For each $\mathbf{u}\in\Sigma^k$, it is clear that 
$$
||D\varphi_\mathbf{u}||^s\ge    M_k^{s-d_q}        ||D\varphi_\mathbf{u}||^{d_q}\ge M_1^{k(s-d_q)}||D\varphi_\mathbf{u}||^{d_q},
$$
and it follows that 
$$
\sum_{\mathbf{u}\in \Sigma^k}||D\varphi_\mathbf{u}||^{s(1-q)}\mu([\mathbf{u}])^q\le M_1^{k(1-q)(s-d_q)}\sum_{\mathbf{u}\in \Sigma^k}||D\varphi_\mathbf{u}||^{d_q(1-q)}\mu([\mathbf{u}])^q.
$$          
Since $P^\mu(d_q,q)=0$, by Lemma \ref{u contraction}, it follows that
$$
\sum_{k=0}^\infty\sum_{\mathbf{u}\in \Sigma^k}||D\varphi_\mathbf{u}||^{s(1-q)}\mu([\mathbf{u}])^q<\infty.
$$

For $q>1$ and $s>d_q$, similarly, we have  
$$
\sum_{\mathbf{u}\in \Sigma^k}||D\varphi_\mathbf{u}||^{s(1-q)}\mu([\mathbf{u}])^q\ge M_1^{k(1-q)(s-d_q)}\sum_{\mathbf{u}\in \Sigma^k}||D\varphi_\mathbf{u}||^{d_q(1-q)}\mu([\mathbf{u}])^q.
$$
Since $P^\mu(s,q)>0$, there exists a number $\gamma>\frac{1}{a^{3q}}$ and a positive integer $K_1$ such that for $k>K_1$
$$
\sum_{\mathbf{u}\in \Sigma^{K_1}}||D\varphi_\mathbf{u}||^{s(1-q)}\mu([\mathbf{u}])^q\ge\gamma.
$$
Since $\mu$ is a Gibbs measure, by \eqref{Gibbs}, it is clear that 
$$
\sum_{\mathbf{u}\in \Sigma^{K_1}}||D\varphi_\mathbf{vu}||^{s(1-q)}\mu([\mathbf{vu}])^q\ge a^{3q}||D\varphi_\mathbf{v}||^{s(1-q)}\mu([\mathbf{v}])^q\sum_{\mathbf{u}\in \Sigma^{K_1}}||D\varphi_\mathbf{u}||^{s(1-q)}\mu([\mathbf{u}])^q. 
$$
Let $\mathcal{C}_0$ be a cut-set such that  for every $\mathbf{u}\in \mathcal{C}_0$, there exist an integer $l\geq 1$  satisfying   $|\mathbf{u}|=l K_1$. It follows that 
$$
 \sum_{\mathbf{u}\in \mathcal{C}_0}||D\varphi_\mathbf{u}||^{s(1-q)}\mu([\mathbf{u}])^q\ge \sum_{\mathbf{u}\in \Sigma^{p{K_1}}}||D\varphi_\mathbf{u}||^{s(1-q)}\mu([\mathbf{u}])^q\ge(a^{3q}\gamma)^l.
$$
Given a cut-set $\mathcal{C}$ with $k_\mathcal{C}\ge l{K_1}$,  for every $\mathbf{u}\in\mathcal{C}$, there exists $\mathbf{v}=\mathbf{u}|k\in\mathcal{C}_0$ with $0\le|\mathbf{u}|-k<{K_1}$, and we have 
$$
M\sum_{\mathbf{u}\in \mathcal{C}}||D\varphi_\mathbf{u}||^{s(1-q)}\mu([\mathbf{u}])^q\ge\sum_{\mathbf{u}\in \mathcal{C}_0}||D\varphi_\mathbf{u}||^{s(1-q)}\mu([\mathbf{u}])^q\ge(a^{3q}\gamma)^l,
$$
where $M=(\frac{\# IC^{s(q-1)}M_1^{s(1-q)}}{a^{3q}})^{K_1}$.  This implies
$$
\lim_{k\to\infty}\min_{\mathcal{C}:k_\mathcal{C}\ge k}\sum_{\mathbf{u}\in\mathcal{C}}||D\varphi_\mathbf{u}||^{s(1-q)}\mu([\mathbf{u}])^q=\infty.
$$

A similar argument holds for $0<q<1$,  and we omit its proof.
\end{proof}

\begin{prop}\label{cut to qici A}
Given ACIFS $\boldsymbol{\Phi}$.  Let $\mu$ be a Gibbs measure on $\Sigma^\infty$ define by \eqref{mu}, and $\mu^\omega$ is the image measure of $\mu$. Then   $\underline{d}_q^*= d_q=\overline{d}_q^*$ for $q>0, q\neq1$.
\end{prop}
\begin{proof}
Since $\underline{d}_q^*\leq \overline{d}_q^*$, we only prove   $\underline{d}_q^*\ge d_q\ge \overline{d}_q^*$ for $q>1$. For each $t<\overline{d}_q^*$, we have $\overline P_\mu(t, q)>0$, and by \eqref{d*}, there exists $\{\delta_m\}$ such that
$$
\sum_{\mathbf{u}\in\mathcal{C}(\delta_m)}||D\varphi_\mathbf{u}||^{t(1-q)}\mu([\bu])^q<e^{-\frac{1}{2}k_{\delta_m} \overline P_\mu(t, q)}.
$$
Since 
$$
\min_{\mathcal{C}:k_\mathcal{C}\ge k_{\delta_m}}\sum_{\mathbf{u}\in\mathcal{C}}||D\varphi_\mathbf{u}||^{s(1-q)}\mu([\mathbf{u}])^q\le\sum_{\mathbf{u}\in\mathcal{C}(\delta_m)}||D\varphi_\mathbf{u}||^{t(1-q)}\mu([\bu])^q , 
$$
it immediately follows that 
$$
\lim_{\delta\to0}\min_{\mathcal{C}:k_\mathcal{C}\ge k_{\delta} }\sum_{\mathbf{u}\in\mathcal{C}}||D\varphi_\mathbf{u}||^{s(1-q)}\mu([\mathbf{u}])^q=0, 
$$
which implies that 
$$
\lim_{k\to\infty}\min_{\mathcal{C}:k_\mathcal{C}\ge k}\sum_{\mathbf{u}\in\mathcal{C}}||D\varphi_\mathbf{u}||^{s(1-q)}\mu([\mathbf{u}])^q=0. 
$$
By Proposition \ref{prop_lme}, it follows that $t\le d_q$ for all $t<\overline{d}_q^*$, and we obtain $\overline{d}_q^*\le d_q$.

For any $t>\underline{d}_q^*$, we have $\underline P_\mu(t, q)<0$,  which means there exists $\{\delta_m\}$ such that
$$
\frac{1}{k_{\delta_m}}\log \sum_{\mathbf{u}\in\mathcal{C}(\delta_m)}\|D\varphi_\mathbf{u}\|^{t(1-q)}\mu([\bu])^q > -\frac{1}{2}\underline P_\mu(t, q),
$$
and
$$
\sum_{\mathbf{u}\in\mathcal{C}(\delta_m)}||D\varphi_\mathbf{u}||^{t(1-q)}\mu([\bu])^q>e^{-\frac{1}{2}k_{\delta_m}  \underline P_\mu(t, q)},
$$
Since for every $\delta_m$ we have
$$
e^\frac{-\underline P_\mu(t, q)k_{\delta_m}}{2}<\sum_{\mathbf{u}\in\mathcal{C}(\delta_m)}||D\varphi_\mathbf{u}||^{t(1-q)}\mu([\bu])^q\le\sum_{k=1}^\infty\sum_{\mathbf{u}\in\Sigma^k}||D\varphi_\mathbf{u}||^{t(1-q)}\mu([\bu])^q.
$$
It follows that
$$
\sum_{k=1}^\infty\sum_{\mathbf{u}\in\Sigma^k}||D\varphi_\mathbf{u}||^{t(1-q)}\mu([\bu])^q=\infty,
$$
and hence $t\ge d_q$. Therefore,  $\underline{d}_q^*\ge d_q$. 
\end{proof}

\begin{proof}[Proof of Theorem \ref{thm_HNCIFS}]
It  follows by Theorem \ref{Dq E} and Proposition \ref{cut to qici A}. 
\end{proof}

\end{document}